\documentclass[11pt]{article}

\usepackage[margin=1.15in]{geometry}
\usepackage{amsmath,amssymb,amsthm,mathtools}
\usepackage[colorlinks=true,citecolor=blue,linkcolor=blue,urlcolor=blue]{hyperref}
\usepackage{microtype}

\usepackage{setspace}
\newtheorem{theorem}{Theorem}[section]

\newtheorem{lemma}[theorem]{Lemma}

\newtheorem{problem}[theorem]{Problem}
\theoremstyle{remark}

\title{Tight degeneracy bounds in online Ramsey games\thanks{Center for Discrete Mathematics, Fuzhou University, Fujian 350108, China.
Email: chenwenfj@163.com, linqizhong@fzu.edu.cn, 3503844078@qq.com. Supported in part by the National Key R\&D Program of China (Grant No. 2023YFA1010202) and the NSFC (No.\ 12571361).}}
\author{
Wen Chen\quad \quad
Qizhong Lin\quad \quad
Shixi Song
}
\date{}
\hypersetup{
  pdftitle={Tight degeneracy bounds in online Ramsey games},
  pdfsubject={Tight degeneracy bounds for online Ramsey games},
  pdfkeywords={online Ramsey theory, degeneracy, survival game, tight bound}
}

\begin{document}
\maketitle

\begin{abstract}
In the \(q\)-color online Ramsey game, Builder and Painter play on an infinite independent set of vertices. At each step, Builder draws an edge and Painter immediately assigns it one of \(q\) colors. Builder's aim is to force a monochromatic copy of a fixed graph \(H\).
We prove that, for every \(q\ge 2\) and \(d\ge 1\), Builder can force a monochromatic copy of any \(d\)-degenerate graph \(H\) while only drawing a \(d\)-degenerate graph. The bound \(d\) is tight, and this resolves in the affirmative a problem of Conlon, Fox and Sudakov.

\end{abstract}

\section{Introduction}

An online Ramsey game is played by Builder and Painter on an infinite
set of initially isolated vertices. Builder exposes one edge at a time,
and Painter immediately colors it with one of \(q\) colors. Given a
target graph \(H\), Builder seeks to force a monochromatic copy of \(H\),
regardless of Painter's choices. This framework originates in work of
Beck~\cite{Beck} and, independently, Friedgut, Kohayakawa, R\"odl,
Ruci\'nski and Tetali~\cite{FKRRT}; see also Kurek and
Ruci\'nski~\cite{KR}.

We study how structural restrictions on Builder's graph affect the
outcome of the game. Grytczuk, Ha\l uszczak and Kierstead~\cite{GHK}
introduced this perspective by requiring the graph formed by the
exposed edges to remain in a prescribed class throughout play.
For a nonnegative integer-valued graph parameter \(\pi\) that is
monotone under taking subgraphs, the general question is whether
a bound on \(\pi(H)\) guarantees a winning strategy in which
\(\pi(G)\) remains bounded solely in terms of \(\pi(H)\) and the
number of colors, where \(G\) is the graph drawn by Builder.
Parameters of interest include chromatic number, degeneracy,
treewidth, thickness and genus. For chromatic number, the optimal
bound is known: every graph \(H\) can be forced while Builder keeps
the chromatic number of the exposed graph at most \(\chi(H)\).
This was proved for two colors in~\cite{GHK} and extended to any
number of colors in~\cite{KK}. Related work considers restrictions
on maximum degree~\cite{BGKMSW}, planarity~\cite{Petrickova}, and
the classes of \(F\)-free and \(K_4\)-minor-free graphs~\cite{CCJO}.

Here we focus on the case where the graph parameter is degeneracy.
A graph \(G\) is \emph{\(d\)-degenerate} if its vertices admit an ordering
\(v_1,\ldots,v_n\) such that each \(v_i\) has at most \(d\) neighbors among
\(v_{i+1},\ldots,v_n\). Such an ordering is called a
\emph{\(d\)-degeneracy ordering} of \(G\).

For \(q\ge2\), let \(f_q(d)\) denote the least \(k\) such that, in the
\(q\)-color online Ramsey game, Builder can force a monochromatic copy of any
\(d\)-degenerate graph while only drawing a \(k\)-degenerate graph.
Necessarily,
\begin{align}\label{tri-low}
 f_q(d)\ge d.
\end{align}
In this language, the problem of~\cite{GHK} is whether \(f_q(d)\) is
finite for all \(q\ge2\) and \(d\ge1\).

For \(d=1\) and \(q\ge2\), it is known~\cite{GHK} that \(f_q(1)=1\);
that is, forests can be forced while Builder draws only forests.
Conlon, Fox and Sudakov~\cite{CFS} resolved the finiteness problem by
proving the upper bound
\[
f_q(d)\le qd-(q-1)
\]
for every \(q\ge2\) and \(d\ge1\). 

For \(d>1\), however, there is still a gap between the lower and upper bounds. Conlon, Fox and Sudakov
then asked whether the lower bound \eqref{tri-low} is always sharp.

\begin{problem}[Conlon, Fox and Sudakov~\cite{CFS}]\label{p:cfs}
For every \(q\ge2\) and \(d\ge1\), is \(f_q(d)=d\)?
\end{problem}

Our main result resolves this problem in the affirmative.

\begin{theorem}\label{thm:main}
For every \(q\ge2\) and \(d\ge1\), in the \(q\)-color online Ramsey game,
Builder can force a monochromatic copy of any \(d\)-degenerate graph \(H\)
while only drawing a \(d\)-degenerate graph. Consequently,
\[
 f_q(d)=d.
\]
\end{theorem}

Our proof starts from the
strategy of Conlon, Fox and Sudakov~\cite{CFS}, in which a witness vertex is
joined to \(p=q(d-1)+1\) vertices. The key ingredient is a discarding argument from~\cite{KK}. After the colors of the \(p\) incident edges are revealed, the
pigeonhole principle guarantees that at least \(d\) of them have the same
color. We keep the corresponding \(d\) neighbors and discard the remaining
\(p-d\) neighbors. Here, ``discard'' means that these vertices are no longer
used in the subsequent construction, while they and all previously drawn
edges remain in the online graph. The discarded vertices can be placed before
the witness vertex in a degeneracy ordering, so the witness has only \(d\)
later neighbors. To iterate this discarding procedure, we use the survival
game theorem of Kierstead and Konjevod~\cite{KK} together with the hypergraph
Ramsey theorem. This allows Builder to maintain a \(d\)-degenerate graph
throughout the game and thereby achieve the optimal bound \(f_q(d)=d\).

\section{Proof of the main theorem}\label{sec:proof}

To implement the discarding argument outlined in the Introduction, we use
the survival game of Kierstead and Konjevod~\cite{KK} together with the
hypergraph Ramsey theorem.

For positive integers \(p,s,t\) with \(s\le p\), the \((p,s,t)\)-\emph{survival game} is played
by Presenter and Chooser on a finite set of vertices, called \emph{live vertices}. In each round
\emph{Presenter} chooses a \(p\)-set \(P\) of live vertices, \emph{Chooser} selects an \(s\)-set
\(X\subseteq P\), the vertices of \(P\setminus X\) are \emph{discarded}, and \(X\) is
recorded as an \(s\)-uniform hyperedge. Presenter wins once the hypergraph on
the live vertices contains a copy of \(K_t^{(s)}\).

\begin{theorem}[Kierstead and Konjevod \cite{KK}]\label{thm:survival}
For all positive integers \(p,s,t\) with \(s\le p\), Presenter has a winning
strategy in the \((p,s,t)\)-survival game from a sufficiently large finite set
of vertices.
\end{theorem}

We shall combine this result with the hypergraph Ramsey number~\cite{ErdosRado,Ramsey}.
For all positive integers \(d,m,q\), let \(r_d(m;q)\) be the least integer \(R\) such that
every \(q\)-coloring of the edges of \(K_R^{(d)}\) contains a monochromatic copy of \(K_m^{(d)}\).

Let \(G\) be the graph consisting of the edges drawn so far. We call a set
\(U\subseteq V(G)\) \emph{active} if no edge has been drawn between two
vertices of \(U\). We say that \((G,U)\) is \emph{\(d\)-admissible} if \(U\) is
active and \(G\) admits a \(d\)-degeneracy ordering in which all vertices of
\(U\) appear last.

The following lemma is the key step of the proof of our main theorem. It describes how Builder
can continue from a \(d\)-admissible pair and obtain a new
\(d\)-admissible pair with a smaller active set, together with many same-colored witnesses.

\begin{lemma}\label{lem:step}
Let \(q,d,m\) be positive integers with \(q\ge2\) and \(m\ge d\).
There exists an integer \(N=N(q,d,m)\) such that the following holds.
For every \(d\)-admissible pair \((G,U)\) with \(|U|\ge N\), Builder has a strategy to guarantee a new \(d\)-admissible pair \((G',V)\) with \(V\subseteq U\), \(|V|=m\) such that for some color \(c\), for every \(D\in\binom{V}{d}\), there is a vertex \(x_D\in U\setminus V\) joined in \(G'\) to every vertex of \(D\) by an edge of color \(c\).
\end{lemma}

\begin{proof}
Set \(p=q(d-1)+1\) and \(R=r_d(m;q)\), and fix a winning Presenter strategy
for the \((p,d,R)\)-survival game on \(N_0\) live vertices, where \(N_0\) is
given by Theorem~\ref{thm:survival}.

When \(d\ge2\), each round discards \(p-d=(q-1)(d-1)\ge1\) vertices, so the
survival game lasts at most \(N_0\) rounds. When \(d=1\), we have \(p=1\) and
\(s=d=1\). Thus in each round the presented set consists of a single vertex,
which Chooser must keep, so no vertex is discarded. In this case we simply
take \(N_0=R\) and present each singleton once. Thus in all cases we may
assume that the fixed Presenter strategy lasts at most \(N_0\) rounds.

Set \(N=2N_0\). Given a \(d\)-admissible pair \((G,U)\) with \(|U|\ge N\),
choose disjoint subsets \(A,B\subseteq U\) with \(|A|=|B|=N_0\).
We shall construct a graph \(G'\) extending \(G\) as follows. Initially set
\(G'=G\). During the simulation below, whenever Builder draws an edge, add
it to \(G'\). At the end of the simulation, \(G'\) is the graph consisting
of \(G\) together with all edges drawn by Builder.

Builder now simulates the fixed Presenter strategy on \(A\).
In round \(i\) of the survival game, let \(P_i=\{v_{i,1},\ldots,v_{i,p}\}\)
be the set prescribed by the Presenter strategy.
Builder chooses an unused vertex \(x_i\in B\) and, over the next \(p\) online rounds, draws the edges
\(x_iv_{i,1},\ldots,x_iv_{i,p}\) one by one, each edge being colored immediately by Painter.
These edges form a star with center \(x_i\) and leaf set \(P_i\).
After all \(p\) edges have been colored,
Builder selects, by the pigeonhole principle, a \(d\)-set \(X_i\subseteq P_i\)
for which all edges from \(x_i\) to \(X_i\) have a common color \(\gamma_i\).

We now feed \(X_i\) to the fixed Presenter strategy as Chooser's response.
This is legitimate because \(X_i\) is a \(d\)-subset of \(P_i\), and Chooser
is allowed to choose any such subset. Although \(X_i\) is induced by
Painter's coloring rather than chosen by an independent Chooser, the
Presenter strategy is winning against every sequence of legal responses.
Thus the simulation remains valid. We record \(X_i\) as a hyperedge and
discard \(P_i\setminus X_i\). If the survival game has not yet been won, we
let \(P_{i+1}\) be the next \(p\)-set prescribed by the fixed Presenter strategy.

The discarded vertices are not used in later survival game rounds. The
edges drawn to them remain in the online graph. Since there are at most \(N_0\)
survival game rounds, \(B\) supplies enough distinct vertices \(x_i\).

Since the Presenter strategy is winning against every sequence of legal
responses, the recorded hypergraph on the live vertices of \(A\) eventually
contains a copy of \(K_R^{(d)}\).
Let \(W\) be the vertex set of this copy. Every \(D\in\binom{W}{d}\) has appeared
as \(X_i\) in some round. For each \(D\in\binom{W}{d}\), choose a round \(i_D\) such that
\(X_{i_D}=D\). Assign \(D\) the color \(\gamma_{i_D}\) and let
\(x_D=x_{i_D}\) be its witness. This defines a \(q\)-coloring of the complete \(d\)-uniform
hypergraph on \(W\). Note that we are only borrowing the colors used by
Painter on the star edges to color the auxiliary hypergraph; no edge of the
original graph is recolored.

By the choice \(R=r_d(m;q)\), there is an \(m\)-set \(V\subseteq W\) whose
\(d\)-subsets all have the same assigned color, say \(c\). Consequently,
for every \(D\in\binom{V}{d}\), all edges from \(x_D\) to \(D\) have color \(c\).
Since \(x_D\in B\) and \(V\subseteq A\), we have \(x_D\in U\setminus V\).
Since \(A\) is active and \(V\subseteq A\), the set \(V\) is active.

It remains to verify the required ordering condition for \((G',V)\).
Suppose that the survival game lasts \(\ell\) rounds. Order \(U\setminus V\)
in the form
\[
 P_1\setminus X_1,\ x_1,\ 
 P_2\setminus X_2,\ x_2,\ \ldots,\ 
 P_\ell\setminus X_\ell,\ x_\ell,
\]
followed by all remaining vertices of \(U\setminus V\), and regard all
vertices of \(V\) as placed last. The order within each of the displayed
sets, and among the remaining vertices, is arbitrary.

In this ordering, every vertex of \(P_i\setminus X_i\) has at most one later
neighbor, namely \(x_i\), while the later neighbors of \(x_i\) are precisely
the vertices of \(X_i\). Every remaining vertex of \(U\setminus V\) has no
later neighbor. Thus every vertex of \(U\setminus V\) has at most \(d\)
later neighbors in \(G'[U]\).

Place the vertices of \(V(G)\setminus U\) first, in their original order,
followed by the vertices of \(U\setminus V\) in the order described above.
Since no new edge is incident with a vertex
of \(V(G)\setminus U\), its number of later neighbors is unchanged. Together
with the fact that \(V\) is active, this shows that \((G',V)\) is
\(d\)-admissible.
\end{proof}

\begin{proof}[Proof of Theorem~\ref{thm:main}]

We now iterate Lemma~\ref{lem:step} to obtain nested active sets, and then
embed \(H\) using the witnesses from stages of a common color.

Let \(H\) be a \(d\)-degenerate graph and set \(n=|V(H)|\).
The assertion is trivial when \(H\) is edgeless. Otherwise $n\ge2$, and we set
\(t=q(n-2)+1\) and \(n_t=d\). For \(i=t,t-1,\ldots,1\), let
\(n_{i-1}=N(q,d,n_i)\), where \(N(q,d,n_i)\) is the integer given by
Lemma~\ref{lem:step} for \(m=n_i\).

Let \(G_0\) be the empty graph on a set \(V_0\) of size \(n_0\).
Starting from the \(d\)-admissible pair \((G_0,V_0)\), Builder applies
Lemma~\ref{lem:step} successively to obtain \(d\)-admissible refinements
\[
 (G_0,V_0)\longrightarrow (G_1,V_1)\longrightarrow\cdots
 \longrightarrow (G_t,V_t),
\]
where \(|V_i|=n_i\) and
\(V_0\supseteq V_1\supseteq\cdots\supseteq V_t\).
For each \(i=1,\ldots,t\), the lemma also gives a color \(c_i\) such that,
for every \(D\in\binom{V_i}{d}\), there is a vertex
\(x_{i,D}\in V_{i-1}\setminus V_i\) joined to every vertex of \(D\) by an
edge of color \(c_i\).
Since \((G_t,V_t)\) is \(d\)-admissible, \(G_t\) is \(d\)-degenerate. Every graph
drawn by Builder during the game is a subgraph of \(G_t\), and hence is also
\(d\)-degenerate.

Since \(t = q(n-2)+1\), the pigeonhole principle guarantees that at least one color, say red (the name is arbitrary), appears at least \(n-1\) times among \(c_1,\dots,c_t\). Select indices \(1\le j_1<\dots<j_{n-1}\le t\) such that each \(c_{j_k}\) is red.

For notational convenience, we relabel the relevant sets by setting
\[
W_0=V_{j_1-1}
\qquad\text{and}\qquad
W_i=V_{j_i}\quad (1\le i\le n-1).
\]
Then
\[
W_0\supseteq W_1\supseteq\cdots\supseteq W_{n-1}.
\]
Moreover, for each \(i=1,\ldots,n-1\) and every
\(D\in\binom{W_i}{d}\), there is a vertex
\(x_D\in W_{i-1}\setminus W_i\) joined in red to every vertex of \(D\).
To see this, note that the witness \(x_D\) for \(D\) in stage \(j_i\) lies in
\(V_{j_i-1}\setminus V_{j_i}\). Since \(j_i-1\ge j_{i-1}\) for \(i>1\),
we have \(V_{j_i-1}\subseteq V_{j_{i-1}}=W_{i-1}\), and therefore
\(x_D\in W_{i-1}\setminus W_i\). For \(i=1\), this is the definition.

Let \(u_n,\ldots,u_1\) be a \(d\)-degeneracy ordering of \(H\). Thus, for every
\(r\), the vertex \(u_r\) has at most \(d\) neighbors among
\(u_1,\ldots,u_{r-1}\). We construct an embedding
\(f\colon V(H)\to V_0\) inductively in the reverse order
\(u_1,\ldots,u_n\).

Choose \(f(u_1)\) arbitrarily in \(W_{n-1}\). Suppose that
\(f(u_1),\ldots,f(u_r)\) have been chosen and all lie in \(W_{n-r}\).
The images of the previously embedded neighbors of \(u_{r+1}\) form a
subset of \(W_{n-r}\) of size at most \(d\). Since \(|W_{n-r}|\ge d\), extend
this subset to some \(D\in\binom{W_{n-r}}{d}\). By the defining property
of \(W_{n-r}\), there is a vertex
\[
f(u_{r+1})=x_D\in W_{n-r-1}\setminus W_{n-r}
\]
joined in red to every vertex of \(D\). Thus all edges from \(u_{r+1}\) to
its previously embedded neighbors are mapped to red edges. Moreover,
\(f(u_{r+1})\) is distinct from the previously chosen vertices, and
\[
f(u_1),\ldots,f(u_{r+1})\in W_{n-r-1}.
\]
The induction therefore continues and produces a red copy of \(H\).
Thus Builder can force a monochromatic copy of \(H\) while
only drawing \(d\)-degenerate graphs. This shows \(f_q(d)\le d\).

Together with the lower bound \eqref{tri-low}, this gives
\(f_q(d)=d\), as desired.
\end{proof}

\section{Concluding remarks}

We have proved that \(f_q(d)=d\) for every \(q\ge2\) and \(d\ge1\).
This determines \(f_q(d)\) exactly and answers Problem~\ref{p:cfs}.

The proof does not aim to optimize the number of edges drawn by Builder.
Its quantitative bounds depend on the survival game theorem, hypergraph
Ramsey numbers, and repeated applications of Lemma~\ref{lem:step}.
It would be interesting to obtain substantially better bounds in terms
of $q$, $d$, and $|V(H)|$, while retaining the optimal degeneracy bound $d$.
In particular, can the use of the hypergraph Ramsey theorem be avoided?

Another natural direction is to study the same question for other graph
parameters. For maximum degree and treewidth, it is known that keeping
Builder's graph within the same parameter bound as the target graph does
not suffice in general.
For maximum degree, Butterfield et al.~\cite{BGKMSW} showed that, for
\(H=K_{\Delta+1}\), any winning strategy for Builder must create a graph of
maximum degree at least \(2\Delta-1\).
For treewidth, Choi et al.~\cite{CCJO} showed that, for
\(H=K_3\), Painter wins the online Ramsey game on graphs of
treewidth at most \(2\). Hence any winning strategy for Builder must
create a graph of treewidth at least \(3\).

This raises the question of how large a bound is necessary. More precisely,
let \(\pi\) denote either maximum degree or treewidth.
For \(q\ge2\) and \(d\ge1\), what is
the least bound \(f_q^\pi(d)\), if finite, such that Builder can force a
monochromatic copy of every graph \(H\) with \(\pi(H)\le d\) in the
\(q\)-color online Ramsey game while keeping \(\pi(G)\le f_q^\pi(d)\)
for the graph \(G\) drawn at every stage?

\section*{Declaration on the use of AI}
In studying the construction of Conlon, Fox and Sudakov~\cite{CFS},
we observed that, for a single star, \(q(d-1)+1\) edges are necessary
and sufficient to guarantee \(d\) edges of one color under an arbitrary
\(q\)-coloring. The key issue was therefore how to retain the desired
\(d\) same-colored neighbors without increasing the degeneracy.
We used generative AI (ChatGPT) to assist in discussing proof strategies.
During these discussions, ChatGPT suggested using the survival game
of Kierstead and Konjevod~\cite{KK} to address this issue.
This suggestion led us to the core idea of the proof of \(f_q(d)=d\).
ChatGPT also assisted with language editing.
All mathematical arguments and conclusions were reviewed and verified
by the authors, who take full responsibility for the content of the paper.

\end{document}